\documentclass[final]{siamltex}

\usepackage[parfill]{parskip}    
\usepackage{amsmath, amssymb, latexsym,enumerate, mathrsfs}
\usepackage{hyperref}
\usepackage{array}
\usepackage{times}

\usepackage{epstopdf}
\usepackage{subfig}
\usepackage{graphicx}

\usepackage{moreverb}
\usepackage{marvosym}
\usepackage{color,soul}
\definecolor{lightblue}{rgb}{.90,.95,1}
\sethlcolor{yellow}

\DeclareMathOperator*{\argmin}{arg\,min}
\DeclareMathOperator*{\argmax}{arg\,max}
\newtheorem{remark}[theorem]{Remark}

\title{On the controlled eigenvalue problem for stochastically perturbed multi-channel systems
\thanks{Received by the editors January 15, 2015. This research was carried-out while the author was working at the University of Notre Dame. The author acknowledges support from the Department of Electrical Engineering, University of Notre Dame.}}

\author{Getachew K. Befekadu\footnotemark[2]}

\begin{document}
\maketitle

\renewcommand{\thefootnote}{\fnsymbol{footnote}}

\footnotetext[2]{NRC/AFRL \& Department of Industrial System Engineering, University of Florida - REEF, 1350 N. Poquito Rd, Shalimar, FL 32579, USA ({\tt gbefekadu@ufl.edu}).}

\renewcommand{\thefootnote}{\arabic{footnote}}

\begin{abstract}
In this brief paper, we consider the problem of minimizing the asymptotic exit rate of diffusion processes from an open connected bounded set pertaining to a multi-channel system with small random perturbations. \,Specifically, we establish a connection between: (i) the existence of an invariant set for the unperturbed multi-channel system w.r.t. certain class of state-feedback controllers; and (ii) the asymptotic behavior of the principal eigenvalues and the solutions of the Hamilton-Jacobi-Bellman (HJB) equations corresponding to a family of singularly perturbed elliptic operators. Finally, we provide a sufficient condition for the existence of a Pareto equilibrium (i.e., a set of optimal exit rates w.r.t. each of input channels) for the HJB equations -- where the latter correspond to a family of nonlinear controlled eigenvalue problems.
\end{abstract}

\begin{keywords} 
Diffusion process, HJB equations, multi-channel system, principal eigenvalue, optimal exit time, small random perturbations
\end{keywords}

\begin{AMS}
37C75, 37H10, 34D15, 34D05, 34D10, 49L25, 34H05
\end{AMS}

\pagestyle{myheadings}
\thispagestyle{plain}
\markboth{GETACHEW K. BEFEKADU}{On the controlled eigenvalue problem}

\section{Introduction}

Consider the following continuous-time multi-channel system\footnote{This work is, in some sense, a continuation of our previous paper \cite{BefAn15a}.}
\begin{equation}
 \dot{x}(t) = A x(t) + \sum\nolimits_{i=1}^n B_i u_i(t), \quad x(0)=x_0, \label{Eq1.1}
\end{equation}
where $x \in \mathcal{X} \subseteq \mathbb{R}^d$ is the state of the system, $u_i \in \mathcal{U}_i \subset \mathbb{R}^{r_i}$ is the control input to the $i$th-channel, and $A \in \mathbb{R}^{d \times d}$ and $B_i \in \mathbb{R}^{d \times r_i}$, for $i=1,2, \ldots, n$, are constant matrices.

Let $D \subset \mathcal{X}$ be an open connected bounded set with smooth boundary. For the multi-channel system in \eqref{Eq1.1}, we consider the following class of state-feedback controllers
\begin{align}
\Gamma \subseteq \Bigl \{\bigl(\gamma_1, \gamma_2, \ldots, \gamma_n\bigr) \in \prod\nolimits_{i=1}^n \mathbb{R}^{r_i}\, \Bigl \lvert\,\Lambda(D) \neq \emptyset  \Bigr\},  \label{Eq1.2}
\end{align}
where $\Lambda(D) \subset D \cup \partial D$ is the maximal invariant set (under the action of $S^t \triangleq \exp \bigl\{\bigl( A + \sum_{i=1}^n B_i \gamma_i\bigr)t\bigr\}$,\, $t \ge 0$) such that
\begin{align}
S^t \Omega = \Omega \subset \Lambda(D), \quad \forall t \ge 0, \label{Eq1.3}
\end{align}
for any set $\Omega$.

In what follows, we provide a connection between the existence of an invariant set for the system $S^t$ in $D \cup \partial D$ and the asymptotic behavior of the principal eigenvalue for singularly perturbed elliptic operator which is associated with the following stochastic differential equation (SDE) 
\begin{equation}
  dX^{\epsilon,\gamma}(t) = \Bigl (A + \sum\nolimits_{i=1}^n B_i \gamma_i\Bigr)X^{\epsilon,\gamma}(t) dt + \sqrt{\epsilon} \sigma\bigl(X^{\epsilon,\gamma}(t)\bigr) dW(t), \,\, X^{\epsilon,\gamma}(0)=x_0, \label{Eq1.4}
\end{equation}
where
\begin{itemize}
\item[-] $X^{\epsilon,\gamma}(\cdot)$ is an $\mathbb{R}^{d}$-valued diffusion process, $\epsilon$ is a small parameter lying in an interval $(0, \epsilon^{\ast})$ (which represents the level of random perturbation in the system), and $\bigl(\gamma_1, \gamma_2, \ldots, \gamma_n\bigr)$ is an $n$-tuple of state-feedbacks from the class $\Gamma$,
\item[-] $\sigma \colon \mathbb{R}^{d} \rightarrow \mathbb{R}^{d \times d}$ is Lipschitz continuous with the least eigenvalue of $\sigma(\cdot)\sigma^T(\cdot)$ uniformly bounded away from zero, i.e., 
\begin{align*}
 \sigma(x)\sigma^T(x)  \ge \kappa I_{d \times d} , \quad \forall x \in \mathbb{R}^{d},
\end{align*}
for some $\kappa > 0$, and
\item[-] $W(\cdot)$ (with $W(0)=0$) is an $d$-dimensional standard Wiener process.
\end{itemize}

For any fixed $\epsilon \in (0,\epsilon^{\ast})$, let $\tau^{\epsilon,\gamma}$ be the first exit time of the diffusion process $X^{\epsilon,\gamma}(t)$ from the set $D$, i.e.,
\begin{equation}
  \tau^{\epsilon,\gamma} = \inf \Bigl \{ t > 0 \, \Bigl \vert \, X^{\epsilon,\gamma}(t) \notin D \Bigr\}. \label{Eq1.5}
\end{equation}
Further, let $P_{\epsilon,0}^{x_0,\gamma}\bigl\{ \mathcal{A}\bigr \}$ and $E_{\epsilon,0}^{x_0,\gamma}\bigl\{\xi \bigr\}$, as usual, denote the probability of an even $\mathcal{A}$ and the expectation of a random variable $\xi$, respectively, for the diffusion process $X^{\epsilon,\gamma}(t)$ starting from $x_0 \in D$. 

Here, it is worth mentioning that, in system reliability analysis and other studies, one often requires to confine a controlled diffusion process $X^{\epsilon,\gamma}(t)$ to a given open connected bounded set $D$ as long as possible. A standard formulation for such a problem is to maximize the mean exit time $E_{\epsilon,0}^{x_0,\gamma}\bigl\{\tau^{\epsilon,\gamma} \bigr\}$ from the set $D$. We also observe that a more suitable objective would be to minimize the asymptotic rate with which the diffusion process $X^{\epsilon,\gamma}(t)$ exits from the set $D$. Further, this suggests minimizing the principal eigenvalue $\lambda_{\epsilon,0}^{\gamma}$ 
\begin{equation}
 \lambda_{\epsilon, 0}^{\gamma} = - \limsup_{t \rightarrow \infty} \frac{1}{t} \log P_{\epsilon,0}^{x_0,\gamma}\Bigl\{\tau^{\epsilon,\gamma} > t \Bigr\}, \label{Eq1.6}
\end{equation}
with respect to certain class of admissible controls (including the above class of state-feedbacks in \eqref{Eq1.2}).\footnote{Recently, the authors in \cite{AraBB16} have provided interesting results on controlled equilibrium selection in stochastically perturbed dynamics, but in a slightly different context.} Note that if the domain $D$ contains an equilibrium point for the deterministic multi-channel system in \eqref{Eq1.1}, when such a system is composed with the $n$-tuple of state-feedbacks from $\Gamma$, then the principal eigenvalue $\lambda_{\epsilon,0}^{\gamma}$, which is associated with the singularly perturbed elliptic operator
\begin{equation}
  -\mathcal{L}_{\epsilon,0}^{\gamma} \bigl(\cdot\bigr) \bigl( x \bigr) = \Bigl \langle \triangledown \bigl(\cdot\bigr),\, \bigl( A + \sum\nolimits_{i=1}^n B_i \gamma_i\bigr)x \Bigr \rangle + \frac{\epsilon}{2} \operatorname{tr} \Bigl \{\sigma(x) \sigma^T(x)  \triangledown^2 \bigl(\cdot\bigr) \Bigr\}, \label{Eq1.7}
\end{equation}
with zero boundary conditions on $\partial D$, satisfies (see also Corollary~\ref{C1})
\begin{align*}
- \log \lambda_{\epsilon,0}^{\gamma} = \epsilon^{-1} r^{\gamma} + o(\epsilon^{-1}) \quad {\text as} \quad \epsilon \rightarrow 0,
\end{align*}
where $r^{\gamma}$ is given by the following
\begin{align}
r^{\gamma} = \limsup_{T \rightarrow \infty} \,\inf_{\varphi(t) \in \Psi} \frac{1}{T} \biggl \{ I_{T}^{\gamma}(\varphi(t)) \, \Bigl \vert \, \varphi(t) \in D \cup \partial D,\,\,  t \in [0, T] \biggr\} \label{Eq1.8}
\end{align}
and the action functional $I_{T}^{\gamma}(\varphi)$, with $(\gamma_1, \gamma_2, \ldots, \gamma_n) \in \Gamma$, is given by
\begin{align}
 I_{T}^{\gamma}(\varphi(t)) =  \frac{1}{2} \int_{0}^{T} \biggl[\frac{d\varphi(t)}{dt} &-  \bigl(A +  \sum\nolimits_{i=1}^n B_i \gamma_i\bigr)\varphi(t) \biggr]^T \Bigl(\sigma(\varphi(t))\sigma^T(\varphi(t))\Bigr)^{-1}\notag \\
&  \quad\quad\quad  \times  \biggl[\frac{d\varphi(t)}{dt} -  \bigl(A +  \sum\nolimits_{i=1}^n B_i \gamma_i\bigr)\varphi(t) \biggr] dt, \label{Eq1.9}
\end{align}
where the set $\Psi$ consisting of all absolutely continuous functions $\varphi \in C_T([0,T], \mathbb{R}^d)$, with $\varphi(0)=x_0$, a compact set in $D$ (e.g., see \cite{VenFre70}, \cite{Ven72}, \cite[Chapter~14]{Fre06} or \cite{FreWe84} for additional discussions).
 
On the other hand, if the maximum invariant set $\Lambda(D)$ for the deterministic multi-channel system in \eqref{Eq1.1}, w.r.t. the state-feedbacks from $\Gamma$, is nonempty. Then, the following asymptotic condition also holds true (see Lemma~\ref{L1} in the Appendix section (cf. \cite[Theorem~2.1]{Kif81}))
\begin{align}
-\lim_{\epsilon \rightarrow 0} \, \limsup_{t \rightarrow \infty} \frac{1} {t} \log P_{\epsilon,0}^{x_0,\gamma} \bigl\{\tau^{\epsilon,\gamma} > t \bigr\} < \infty, \,\,x_0 \in D. \label{Eq1.10}
\end{align}
Moreover, the principal eigenvalue $\lambda_{\epsilon,0}^{\gamma}$ turns out to be the boundary value between those $R < r^{\gamma}$ for which $E_{\epsilon,0}^{x_0,\gamma}\bigl\{\exp(\epsilon^{-1} R \tau^{\epsilon,\gamma})\bigr\} < \infty$ and those $R > r^{\gamma}$ for which $E_{\epsilon,0}^{x_0,\gamma}\bigl\{\exp(\epsilon^{-1} \\R \tau^{\epsilon,\gamma})\bigr\} = \infty$ (see also \cite[pp.~373--382]{Fre06} for additional discussions). Note that estimating the asymptotic exit rate with which the controlled-diffusion process $X^{\epsilon,\gamma}(t)$ exits from the domain $D$ is also related to a singularly perturbed eigenvalue problem. For example, the asymptotic behavior for the principal eigenvalue corresponding to the following eigenvalue problem
\begin{align}
\left.\begin{array}{c}
\mathcal{L}_{\epsilon,0}^{\gamma^{\ast}}\psi_{\gamma^{\ast}}^{\ast}(x) + \lambda_{\epsilon,0}^{\gamma^{\ast}}\,\psi_{\gamma^{\ast}}^{\ast}(x) = 0, \quad \forall x \in D\\
\quad\quad ~ \psi_{\gamma^{\ast}}^{\ast}(x) = 0, \quad \forall x \in \partial D\\
\end{array} \right \} \label{Eq1.11}
\end{align}
where $\psi_{\gamma^{\ast}}^{\ast} \in W_{loc}^{2, p}(D) \cap C(D \cup \partial D)$, for $p > 2$, with $\psi_{\gamma^{\ast}}^{\ast} > 0$ on $D$, has been well studied in the past (e.g., see \cite{Day83} in the context of an asymptotic behavior for the principal eigenfunction; and see \cite{Day87} or \cite{BiBor09} in the context of an asymptotic behavior for the equilibrium density).

Before concluding this section, let us introduce the following definition for minimal action state-feedbacks, w.r.t. the action functional $I_T^{\gamma}(\varphi)$, which is useful for the development of our main result.
\begin{definition} \label{D1}
The $n$-tuple $(\gamma_1^{\ast}, \gamma_2^{\ast}, \ldots, \gamma_n^{\ast}) \in \Gamma$ is said to be minimal action state-feedbacks if 
\begin{align}
(\gamma_1^{\ast}, \gamma_2^{\ast}, \ldots, \gamma_n^{\ast}) \in \argmin I_T^{\gamma}(\varphi). \label{Eq1.12}
\end{align}
\end{definition}

In the following section, we present our main results -- where we establish a connection between the asymptotic exit rate with which the controlled diffusion process (w.r.t. each of the input channels) exits from the set $D$ and the asymptotic behavior of principal eigenvalues for a family of singularly perturbed elliptic operators with zero boundary condition on $\partial D$. Later, such a formulation allows us to provide a sufficient condition for the existence of a Pareto equilibrium (i.e., a set of optimal exit rates w.r.t. each of the input channels) for the HJB equations -- where the latter correspond to a family of nonlinear controlled eigenvalue problems (e.g., see \cite{QuaSi08} or \cite[Chapter~8]{GilTr01} for additional discussions on eigenvalue problems).  

\section{Main Results}
In this section, we consider a family of SDEs (w.r.t. each of input channels)
\begin{align}
  dX^{\epsilon, i}(t) = \Bigl (A + \sum\nolimits_{j \neq i} B_j \gamma_j\Bigr)X^{\epsilon, i}(t) dt + B_i u_i(t)dt + \sqrt{\epsilon} \sigma\bigl(X^{\epsilon, i}(t)\bigr) dW(t), \notag \\ X^{\epsilon,i}(0)=x_0, \quad  i =1, 2, \ldots, n, \label{Eq2.1}
\end{align}
where $u_i(\cdot)$ is a $\mathcal{U}_i$-valued measurable control process to the $i$th-channel (i.e., an admissible control from the set \,$\mathcal{U}_i \subset \mathbb{R}^{r_i}$) such that for all $t > s$, $W(t) - W(s)$ is independent of $u_i(\nu)$ for $\nu \le s$ (nonanticipativity condition) and
\begin{align*}
E \int_{s}^{t} \vert u_i(\tau)\vert^2 d\tau < \infty, \quad \forall t \ge s, 
\end{align*}
for $i=1, 2, \dots, n$.

Next, let 
\begin{align}
 \lambda_{\epsilon}^{u} =  \bigl(\lambda_{\epsilon,1}^{u_1},\,  \lambda_{\epsilon,2}^{u_2},\, \ldots, \,\lambda_{\epsilon,n}^{u_n}\bigr),  \label{Eq2.2}
\end{align}
with 
\begin{align*}
 \lambda_{\epsilon, i}^{u_i} =  -\limsup_{t \rightarrow \infty} \frac{1}{t} \log P_{\epsilon,i}^{x_0,u_i}\bigl\{\tau_i^{\epsilon} > t \bigr\}, \quad i=1,2, \ldots, n,
 \end{align*}
where the probability $P_{\epsilon,i}^{x_0,u_i}$ is conditioned on the initial point $x_0 \in D$ and the admissible control $u_i \in \mathcal{U}_i$. Moreover, $\tau_i^{\epsilon}$ is the first exit time for the diffusion process $X^{\epsilon, i}(t)$ from the set $D$, i.e., $\tau_i^{\epsilon} = \inf \bigl \{ t > 0 \, \bigl \vert \, X^{\epsilon, i}(t) \notin D \bigr\}$ for $i=1,2, \ldots, n$.

Further, let us introduce the following set
\begin{equation}
 \Sigma^{\epsilon} =  \Bigl\{ \lambda_{\epsilon}^{u} \in \mathbb{R}_{+}^n \, \Bigl \vert \, \bigl(\gamma_1, \gamma_2, \ldots, \gamma_n \bigr) \in \Gamma \Bigr\}.  \label{Eq2.3}
\end{equation}

\begin{remark}
Notice that the set $\Sigma^{\epsilon}$ (w.r.t. the class of state-feedbacks $\Gamma$) is a closed subset of $\mathbb{R}^n$.
\end{remark}

Define the partial ordering $\prec$ on $\Sigma^{\epsilon}$ by $\lambda_{\epsilon}^{u'} \prec \lambda_{\epsilon}^{u''}$, i.e.,
\begin{equation}
   \bigl(\lambda_{\epsilon,1}^{u_1'}, \lambda_{\epsilon,2}^{u_2'},  \ldots, \lambda_{\epsilon,n}^{u_n'} \bigr) \prec \bigl(\lambda_{\epsilon,1}^{u_1''}, \lambda_{\epsilon,2}^{u_2''},  \ldots, \lambda_{\epsilon,n}^{u_n''} \bigr),  \label{Eq2.4}
\end{equation}
if $\lambda_{\epsilon,i}^{u_i'} \le \lambda_{\epsilon,i}^{u_i''}$ for all $i=1,2, \ldots, n$, with strict inequality for at least one $i \in \{1,2, \dots, n\}$. Further, we say that $\lambda_{\epsilon}^{u^{\ast}} \in \Sigma^{\epsilon}$ is a Pareto equilibrium (w.r.t. the class of state-feedbacks $\Gamma$) if there is no $\lambda_{\epsilon}^u \in \Sigma^{\epsilon}$ for which $\lambda_{\epsilon}^u \prec \lambda_{\epsilon}^{u^{\ast}}$.

Then, we have the following proposition that provides a sufficient condition for the existence of a Pareto equilibrium $\lambda_{\epsilon}^{u^{\ast}} \in \Sigma^{\epsilon}$ (i.e., a set of optimal exit rates w.r.t. each of input channels).

\begin{proposition}\label{P1}
Suppose that the statement in Definition~\ref{D1} holds true at least for one $n$-tuple of state-feedbacks $(\gamma_1^{\ast}, \gamma_2^{\ast}, \ldots, \gamma_n^{\ast}) \in \Gamma$. Then, there exists a Pareto equilibrium $\lambda_{\epsilon}^{u^{\ast}} \in \Sigma^{\epsilon}$ (i.e., a set of optimal exit rates w.r.t. each of input channels) such that 
\begin{equation}
   \bigl(\lambda_{\epsilon,1}^{u_1^{\ast}}, \lambda_{\epsilon,2}^{u_2^{\ast}}, \ldots, \lambda_{\epsilon,n}^{u_n^{\ast}} \bigr) \prec \bigl(\lambda_{\epsilon,1}^{u_1}, \lambda_{\epsilon,2}^{u_2},  \ldots, \lambda_{\epsilon,n}^{u_n} \bigr) \quad \text{on} \quad \Sigma^{\epsilon},  \label{Eq2.5}
\end{equation}
where the principal eigenvalues $\lambda_{\epsilon,i}^{u_i^{\ast}}$ are the unique solutions for the HJB equations corresponding to the following family of nonlinear controlled eigenvalue problems
\begin{align}
\left.\begin{array}{c}
\max_{u_i \in \mathcal{U}_i} \Bigl \{\mathcal{L}_{\epsilon,i}^{u_i}\,\psi_{u_i}^{\ast}(x, u_i) + \lambda_{\epsilon,i}^{u_i^{\ast}}\,\psi_{u_i}^{\ast}(x) = 0 \Bigr \}, \quad \forall x \in D\\
\quad\quad ~ \psi_{u_i}^{\ast}(x) = 0, \quad \forall x \in \partial D\\
\end{array} \right. \label{Eq2.6}
\end{align}
with $\psi_{u_i}^{\ast} \in W_{loc}^{2, p}(D) \cap C(D \cup \partial D)$, for $p > 2$, with $\psi_{u_i}^{\ast}> 0$ on $D$, and
\begin{equation}
  -\mathcal{L}_{\epsilon,i}^{u_i} \bigl(\cdot\bigr) \bigl( x \bigr) = \Bigl \langle \triangledown \bigl(\cdot\bigr),\, \bigl( A + \sum\nolimits_{j \neq i} B_j \gamma_j^{\ast} \bigr)x + B_i u_i\Bigr \rangle + \frac{\epsilon}{2} \operatorname{tr} \Bigl \{\sigma(x) \sigma^T(x)  \triangledown^2 \bigl(\cdot\bigr) \Bigr\}, \label{Eq2.7}
\end{equation}
 for $i =1,2, \ldots, n$.
\end{proposition}

\begin{proof} 
Suppose there exists an $n$-tuple of state-feedbacks $(\gamma_1^{\ast}, \gamma_2^{\ast}, \ldots, \gamma_n^{\ast}) \in \Gamma$ that satisfies the statement in Definition~\ref{D1}. Then, our first claim for $\psi_{u_i}^{\ast} \in W_{loc}^{2,p} \bigl(D\bigr) \cap C\bigl(D \cup \partial D \bigr)$, with $p>2$, follows from Equation~\eqref{Eq2.6} (cf. \cite[Theorems~1.1, 1.2 and 1.4]{QuaSi08}). That is, if $u_i^{\ast}$ for $i \in \{1,2, \ldots, N\}$, is a measurable selector of $\argmax \bigl\{\mathcal{L}_{\epsilon, i}^{u_i} \psi_{u_i}^{\ast} \bigl(x\bigr) \bigr\}$, with $x \in D$. Then, by the uniqueness claim for the eigenvalue problem (cf. Equation~\eqref{Eq1.11}), we have
\begin{align}
\lambda_{\epsilon, i}^{u_i^{\ast}} = - \limsup_{t \rightarrow \infty} \frac{1} {t} \log P_{\epsilon,i}^{x_0, u_i^{\ast}} \Bigl\{\tau_i^{\epsilon} > t \Bigr\}, \label{Eq2.8}
\end{align}
where the probability $P_{\epsilon,i}^{x_0, u_i^{\ast}}$ is conditioned on $x_0$, $u_i^{\ast}$ and $\bigl(\gamma_1^{\ast}, \dots, \gamma_{i-1}^{\ast}, \gamma_{i+1}^{\ast}, \ldots, \gamma_n^{\ast}\bigr)$. Moreover, for any other admissible controls $v_i \in \mathcal{U}_i$, we have
\begin{align}
\mathcal{L}_{\epsilon, i}^{u_i^{\ast}} \psi_{u_i}^{\ast} \bigl(x, v_i \bigr) + \lambda_{\epsilon, i}^{u_i^{\ast}} \psi_{u_i}^{\ast} \bigl(x\bigr) \le 0, \quad \forall t \ge 0. \label{Eq2.9}
\end{align}
Let $Q \subset \mathbb{R}^d$ be a smooth bounded open domain containing $D \cup \partial D$. Let $\hat{\psi}_{v_i}$ and  $\hat{\lambda}_{\epsilon,i}^{v_i}$ be the principal eigenfunction-eigenvalue pairs for the eigenvalue problem of $\mathcal{L}_{\epsilon, i}^{v_i}$ on $\partial Q$. Further, let 
\begin{align}
\hat{\tau}_i^{\epsilon} = \inf \Bigl\{ t > 0 \, \bigl\vert \, X^{\epsilon, i}(t) \notin Q \Bigr\}. \label{Eq2.10}
\end{align}
Then, under $v_i$, for $i \in \{1,2, \ldots, N\}$ and for some $\hat{x} \in D$, we have  
\begin{align}
 \hat{\psi}_{v_i} \bigl(\hat{x}\bigr) &\ge E_{\epsilon,i}^{x_0, v_i} \Biggl\{\exp\bigl( \hat{\lambda}_{\epsilon, i}^{v_i} t \bigr) \,\hat{\psi}_{v_i} \bigl(x(t)\bigr) \mathbf{1}\Bigl\{\hat{\tau}_i^{\epsilon} > t\Bigr\} \Biggr\},  \notag\\
  & \ge  \inf_{x  \in D}\Bigl\vert \hat{\psi}_{v_i} \bigl(x(t)\bigr) \Bigr\vert  \exp\bigl(\hat{\lambda}_{\epsilon,i}^{v_i}\, t\bigr) P_{\epsilon,i}^{x_0, u_i^{\ast}}\Bigl\{\tau_i^{\epsilon} > t\Bigr\}. \label{Eq2.11}
\end{align}
Leading to
\begin{align}
\hat{\lambda}_{\epsilon, i}^{v_i} \ge - \limsup_{t \rightarrow \infty} \frac{1} {t} \log P_{\epsilon,i}^{x_0, u_i^{\ast}} \Bigl\{\tau_i^{\epsilon} > t \Bigr\}. \label{Eq2.12}
\end{align}
Letting $Q$ shrink to $D$ and using Proposition~4.10 of \cite{QuaSi08}, then we have $\hat{\lambda}_{\epsilon, i}^{v_i} \rightarrow \lambda_{\epsilon,i}^{u_i^{\ast}}$.Thus, we have
\begin{align}
\lambda_{\epsilon,i}^{u_i^{\ast}} \le - \limsup_{t \rightarrow \infty} \frac{1} {t} \log P_{\epsilon,i}^{x_0, u_i^{\ast}} \Bigl\{\tau_i^{\epsilon} > t \Bigr\}. \label{Eq2.13}
\end{align}
Combining with Equation~\eqref{Eq2.8}, this establishes the optimality of $u_i^{\ast}$ and the fact that $\lambda_{\epsilon,i}^{u_i^{\ast}}$ is the optimal exit rate.

Conversely, let $v_i^{\ast}$ be an admissible optimal control, then we have
\begin{align}
\mathcal{L}_{\epsilon, i}^{u_i^{\ast}} \psi_{v_i^{\ast}}^{\ast} \bigl(x, v_i^{\ast}\bigr) + \hat{\lambda}_{\epsilon, i}^{v_i^{\ast}} \, \psi_{v_i^{\ast}}^{\ast} \bigl(x \bigr) = 0 \label{Eq2.14}
\end{align}
and
\begin{align}
\mathcal{L}_{\epsilon, i}^{u_i^{\ast}} \psi_{u_i^{\ast}}^{\ast} \bigl(x, v_i^{\ast}\bigr) + \lambda_{\epsilon, i}^{u_i^{\ast}} \, \psi_{u_i^{\ast}}^{\ast} \bigl(x \bigr) \le 0, \quad \forall t > 0, \label{Eq2.15}
\end{align}
with $\hat{\lambda}_{\epsilon, i}^{v_i^{\ast}} = \lambda_{\epsilon,i}^{u_i^{\ast}}$, for $i = 1, 2, \ldots, n$.

Further, notice that $\psi_{u_i^{\ast}}^{\ast}$ is a scalar multiple of $\psi_{v_i^{\ast}}^{\ast}$ and, at some $\hat{x} \in D$ (cf. \cite[Theorem~1.4(a)]{QuaSi08}). Then, we see that $v_i^{\ast}$ is also a maximizing measurable selector in Equation~\eqref{Eq2.6}.

On the other hand, for a fixed small $\epsilon$, define the following continuous functional (i.e., a utility function over a closed set $\Sigma^{\epsilon} \in \mathbb{R}_{+}^n$ that can be convexfied)
\begin{align}
\Sigma^{\epsilon} \ni \lambda_{\epsilon}^{u} \mapsto U^{\epsilon} \bigl( \lambda_{\epsilon}^{u} \bigr) \triangleq \bigl\langle \omega,\, \lambda_{\epsilon}^{u}\bigr\rangle \in \mathbb{R}, \label{Eq2.16}
\end{align}
where $\omega_i > 0$ for $i=1,2,\ldots, n$. 

Note that the utility function $U^{\epsilon}$ satisfies the property $U^{\epsilon}\bigl(\lambda_{\epsilon}^{u'}\bigr) < U^{\epsilon}\bigl(\lambda_{\epsilon}^{u''}\bigr)$, whenever $\lambda_{\epsilon}^{u'} \prec \lambda_{\epsilon}^{u''}$ on $\Sigma^{\epsilon}$ w.r.t. the class of state-feedbacks $\Gamma$. Then, from the Arrow-Barankin-Blackwell theorem (e.g., see \cite{ArroBB53}), one can see that the set in
\begin{align}
\biggl \{ \lambda_{\epsilon}^{u} \in \Sigma^{\epsilon} \, \Bigl \vert \,  \exists \,\omega_i >0, \,\, i=1,2, \ldots, n, \,\, \min_{\lambda_{\epsilon}^{u} \in \Sigma^{\epsilon}} \bigl\langle \omega,\, \lambda_{\epsilon}^{u}\bigr\rangle = \bigl\langle \omega,\, \lambda_{\epsilon}^{u^{\ast}} \bigr\rangle  \biggr\} \label{Eq2.17}
\end{align}
is dense in the set of all Pareto equilibria. This further implies that, for any choice of $\omega_i > 0$, $i=1,2, \ldots, n$, the minimizer $\bigl\langle \omega,\, \lambda_{\epsilon}^{u}\bigr\rangle = \sum_{i=1}^n \omega_i \lambda_{\epsilon,i}^{u_i}$ over $\Sigma^{\epsilon}$ satisfies the Pareto equilibrium condition w.r.t. some minimal action state-feedbacks $(\gamma_1^{\ast}, \gamma_2^{\ast}, \ldots, \gamma_n^{\ast}) \in \Gamma$. This completes the proof of Proposition~\ref{P1}.
\end{proof}

We conclude this section with the following corollary that gives a lower-bound for $r^{\gamma^{\ast}}$ in \eqref{Eq1.8} (w.r.t. some minimal action state-feedbacks $(\gamma_1^{\ast}, \gamma_2^{\ast}, \ldots, \gamma_n^{\ast}) \in \Gamma$) (cf. \cite[Corollary~11.2,~pp.~377]{Fre06})
\begin{corollary} \label{C1} 
Suppose that the statement in Proposition~\ref{P1} holds. Then, $r^{\gamma^{\ast}}$ in \eqref{Eq1.8}, w.r.t. the minimal action state-feedbacks $(\gamma_1^{\ast}, \gamma_2^{\ast}, \ldots, \gamma_n^{\ast}) \in \Gamma$, satisfies\footnote{Notice that
\begin{align*}
r^{\gamma^{\ast}} = \limsup_{T \rightarrow \infty} \,\inf_{\varphi(t) \in \Psi} \frac{1}{T} \biggl \{ I_{T}^{\gamma^{\ast}}(\varphi(t)) \, \Bigl \vert \, \varphi(t) \in D \cup \partial D,\,\,  t \in [0, T] \biggr\},
\end{align*}
where
\begin{align*}
 I_{T}^{\gamma^{\ast}}(\varphi(t)) =  \frac{1}{2} \int_{0}^{T} \biggl[\frac{d\varphi(t)}{dt} &-  \bigl(A +  \sum\nolimits_{i=1}^n B_i \gamma_i^{\ast}\bigr)\varphi(t) \biggr]^T \Bigl(\sigma(\varphi(t))\sigma^T(\varphi(t))\Bigr)^{-1}\notag \\
&  \quad\quad\quad  \times  \biggl[\frac{d\varphi(t)}{dt} -  \bigl(A +  \sum\nolimits_{i=1}^n B_i \gamma_i^{\ast}\bigr)\varphi(t) \biggr] dt. 
\end{align*}}
\begin{align}
  r^{\gamma^{\ast}} \ge \max_{i \in \{1,2,\ldots, n\}}  r^{u_i^{\ast}},  \label{Eq2.18}
\end{align}
where $r^{u_i^{\ast}}$ is given by
\begin{align}
  r^{u_i^{\ast}} = - \epsilon \log \lambda_{\epsilon, i}^{u_i^{\ast}}  \quad {\text as} \quad \epsilon \rightarrow 0. \label{Eq2.19}
\end{align}
\end{corollary}

\begin{remark} 
Here, we remark that, for each $i \in \{1, 2, \ldots, n\}$, we have
\begin{align*}
 \lambda_{\epsilon,0}^{\gamma} -  \lambda_{\epsilon, i}^{u_i^{\ast}} \ge 0, \quad \forall \epsilon \in (0, \epsilon^{\ast}), \quad (\gamma_1, \gamma_2, \ldots, \gamma_n) \in \Gamma,
\end{align*}
which suggests that the above corollary is useful for selecting the most appropriate minimal action state-feedbacks from $\Gamma$ that confine the diffusion process $X^{\epsilon}(t)$ to the prescribed set $D$ for a longer duration (cf. Equations~\eqref{Eq1.6} and \eqref{Eq2.2}).
\end{remark}

\section*{Appendix}
The following lemma (whose proof is an adaptation of \cite[Theorem~2.1]{Kif81}) provides a condition under which the maximal invariant set $\Lambda(D)$ is nonempty.
\begin{lemma}\label{L1}
\begin{enumerate} [(a)]
\item If for some $x_0 \in D$,
\begin{equation}
  - \limsup_{\epsilon \rightarrow 0} \limsup_{t \rightarrow \infty} \frac{1}{t} \log P_{\epsilon,0}^{x_0,\gamma}\Bigl\{\tau^{\epsilon,\gamma} > t \Bigr\} < \infty,  \label{EqA.1} \tag{A.1}
\end{equation}
then the maximal invariant set $\Lambda(D)$ is nonempty.
\item If for some $x_0 \in D$,
\begin{equation}
  \limsup_{\epsilon \rightarrow 0} \log E_{\epsilon,0}^{x_0,\gamma}\Bigl\{\tau^{\epsilon,\gamma} \Bigr\} = \infty,  \label{EqA.2} \tag{A.2}
\end{equation}
then the maximal invariant set $\Lambda(D)$ is nonempty.
\end{enumerate}
\end{lemma}

\begin{proof} Suppose that the maximum closed invariant set $\Lambda(D)$ is empty (i.e., the invariant set for $S^t$ in $D \cup \partial D$ is empty). Then, there exists an open bounded domain $\tilde{D} \supset D \cup \partial D$ such that the corresponding set $\Lambda(\tilde{D})$ is also empty.

Note that it is easy to check that if $D_{2} \subset D_{1}$, then $\Lambda(D_{2}) \subset \Lambda(D_{1})$. Take the following sequence $\bigl\{D_{m}\bigr\}$ of open domains such that
\begin{align}
D_{1} \supset D_{2} \supset D_{3} \supset \cdots \quad \text{and} \quad \bigcap\nolimits_{m \ge 1} D_{m} = D \cup \partial D. \label{EqA.3} \tag{A.3}
\end{align}
If $\Lambda(D_{m})\neq \emptyset$ for all $m \ge 1$, then
\begin{align}
\Lambda = \bigcap\nolimits_{m \ge 1} \Lambda(D_{m}). \label{EqA.4} \tag{A.4}
\end{align}
Moreover, since $\Lambda(D_{m})$ is closed, we have
\begin{align}
\Lambda(D_{1}) \supset \Lambda(D_{2}) \supset \Lambda(D_{3}) \supset  \cdots. \label{EqA.5} \tag{A.5}
\end{align}
Note that $\Lambda$ is an invariant closed set for the unperturbed system
\begin{align*}
\dot{x}^{\gamma}(t) = \Bigl (A + \sum\nolimits_{i=1}^n B_i \gamma_i\Bigr)x^{\gamma}(t), \quad (\gamma_1, \gamma_2, \ldots, \gamma_n) \in \Gamma, \quad x_0^{\gamma}=x_0
\end{align*}
and $\Lambda \supset D \cup \partial D$. Thus, $\emptyset \neq \Lambda \subset \Lambda(D)$. This contradicts our earlier assumption. Then, for some $m_0 \ge 1$, we have 
\begin{align}
\Lambda(D_{m_0}) = \emptyset. \label{EqA.6} \tag{A.6}
\end{align}
Let $\tilde{D} = D_{m_0}$, for any $\epsilon \in (0,\epsilon^{\ast})$ and $x_0 \in \tilde{D} \cup \partial \tilde{D}$, introduce the following
\begin{align}
\tau_{\tilde{D}}^{\epsilon,\gamma} = \inf \Bigl\{ t > 0 \, \bigl\vert \, \exp \Bigl\{\Bigl( A + \sum\nolimits_{i=1}^n B_i \gamma_i\Bigr)t\Bigr\} x_0 \notin \tilde{D} \cup \partial \tilde{D} \Bigr\}. \label{EqA.7} \tag{A.7}
\end{align}
Then, we can show that $\tau_{\tilde{D}}^{\epsilon,\gamma} < \infty$ for any $x_0 \in \tilde{D} \cup \partial \tilde{D}$. 

Notice that, if $\tau_{\tilde{D}}^{\epsilon,\gamma} =\infty$, then $\tilde{D} \cup \partial \tilde{D} \ni \exp \Bigl\{\Bigl( A + \sum\nolimits_{i=1}^n B_i \gamma_i\Bigr)t\Bigr\}x_0$ for all $t \ge 0$. Then, we have the following
\begin{align}
\Bigl\{ \exp \Bigl\{\Bigl( A + \sum\nolimits_{i=1}^n B_i \gamma_i\Bigr)t\Bigr\}x_0, \,\, t \ge 0 \Bigr\} \subset \Lambda(\tilde{D}) = \emptyset, \label{EqA.8} \tag{A.8}
\end{align}
which show that $\tau_{\tilde{D}}^{\epsilon,\gamma}$ is finite.

Note that, from upper-semicontinuity of $\tau_{\tilde{D}}^{\epsilon,\gamma}$, we have
\begin{align}
T = \sup_{x_0 \in \tilde{D} \cup \partial \tilde{D}} \tau_{\tilde{D}}^{\epsilon,\gamma} < \infty. \label{EqA.9} \tag{A.9}
\end{align}
Moreover, for any $\delta > 0$, let
\begin{align}
\lim_{\epsilon \rightarrow 0} \sup_{x_0 \in D} E_{\epsilon,0}^{x_0,\gamma} \Bigl\{ \operatorname{dist} \Bigl(X^{\epsilon,\gamma}(t), \exp \Bigl\{\Bigl( A + \sum\nolimits_{i=1}^n B_i \gamma_i\Bigr)t\Bigr\}x_0 \Bigr) > \delta \Bigr\} = 0, \,\, t \ge 0. \label{EqA.10} \tag{A.10}
\end{align}
From Equations~\eqref{EqA.7}--\eqref{EqA.10}, we have
\begin{align}
\sup_{x_0 \in D} P_{\epsilon,0}^{x_0,\gamma} \Bigl\{ \tau^{\epsilon,\gamma} > T \Bigr\} \rightarrow 0 \quad \text{as} \quad \epsilon \rightarrow 0. \label{EqA.11} \tag{A.11}
\end{align}
Then, using the Markov property, we have
\begin{align}
 P_{\epsilon,0}^{x_0,\gamma} \Bigl\{ \tau^{\epsilon,\gamma} > \ell T \Bigr\} &= E_{\epsilon,0}^{x_0,\gamma} \chi_{\tau^{\epsilon,\gamma} > T} E_{\epsilon,0}^{x_0,\gamma} \chi_{\tau^{\epsilon,\gamma} > T} \cdots E_{\epsilon,0}^{x_0,\gamma} \chi_{\tau^{\epsilon,\gamma} > T}, \notag \\
 & \le \Bigl(\sup_{x_0 \in D} P_{\epsilon,0}^{x_0,\gamma} \Bigl\{ \tau^{\epsilon,\gamma} > T \Bigr\} \Bigr)^{\ell}, \label{EqA.12} \tag{A.12}
\end{align}
where $\chi_{\mathcal{A}}$ is the indicator for the event $\mathcal{A}$.

Since $P_{\epsilon,0}^{x_0,\gamma} \Bigl\{ \tau^{\epsilon,\gamma} > t \Bigr\}$ decreases in $t$, then we have
\begin{align}
\limsup_{t \rightarrow \infty} \frac{1}{t} \log P_{\epsilon,0}^{x_0,\gamma} \Bigl\{ \tau^{\epsilon,\gamma} > t \Bigr\} \le  \frac{1}{T} \log \Bigl(\sup_{x_0 \in D} P_{\epsilon,0}^{x_0,\gamma} \Bigl\{ \tau^{\epsilon,\gamma} > T \Bigr\}\Bigr). \label{EqA.13} \tag{A.13}
\end{align}
Taking into account Equation~\eqref{EqA.11}, then, for any $x_0 \in D$, we have the following
\begin{align}
  \limsup_{t \rightarrow \infty} \frac{1}{t} \log P_{\epsilon,0}^{x_0,\gamma} \Bigl\{\tau^{\epsilon,\gamma} > t \Bigr\}  \rightarrow -\infty \quad \text{as} \quad \epsilon \rightarrow 0. \label{EqA.14} \tag{A.14}
\end{align}
Hence, our assumption that $\Lambda(D) = \emptyset$ is inconsistent.

To proof part (ii), notice that
\begin{align}
 E_{\epsilon,0}^{x_0,\gamma} \Bigl\{ \tau^{\epsilon,\gamma} \Bigr\} \le T\sum\nolimits_{\ell=1}^{\infty} P_{\epsilon,0}^{x_0,\gamma} \Bigl\{ \tau^{\epsilon,\gamma} > (\ell-1)T \Bigr\}. \label{EqA.15}\tag{A.15}
\end{align}
Assumption  $\Lambda(D )= \emptyset$ gives, in view of Equations~\eqref{EqA.11}, \eqref{EqA.12} and \eqref{EqA.13} for sufficiently small $\epsilon > 0$ and for any $x_0 \in D$,
\begin{align}
 E_{\epsilon,0}^{x_0,\gamma} \Bigl\{ \tau^{\epsilon,\gamma} \Bigr\} &\le T\sum\nolimits_{\ell=1}^{\infty} \Bigl[\sup_{x_0 \in D} P_{\epsilon,0}^{x_0,\gamma} \Bigl\{ \tau^{\epsilon,\gamma} > T \Bigr\}\Bigr]^{\ell-1} \notag \\
 &< \infty, \label{EqA.16} \tag{A.16}
\end{align}
which contradicts with Equation~\eqref{EqA.2}. This completes the proof of Lemma~\ref{L1}. 
\end{proof}


\begin{thebibliography}{99}

\bibitem{AraBB16}
{\sc A.~Arapostathis, A. Biswas and V. S. Borkar}, {\em Controlled equilibrium selection in stochastically perturbed dynamics},
\href{http://arxiv.org/abs/1504.04889v2}{arXiv:1504.04889v2} \href{http://arxiv.org/archive/math.CT}{[math.CT]}, August 2016.

\bibitem{ArroBB53}
{\sc K.~J. Arrow, E. W. Barankin and D. Blackwell}, {\em Admissible points of convex sets}, in H. W. Kuhn and A. W. Tucker (eds.), {\em Contributions to the theory of games,} vol. II, 
Princeton, NJ, (1953), pp.~87--91.
 
\bibitem{BefAn15a}
{\sc G.~K. Befekadu and P. J. Antsaklis}, {\em On the asymptotic estimates for exit probabilities and minimum exit rates of diffusion processes pertaining to a chain of distributed control systems}, SIAM J. Contr. Optim., 53 (2015), pp.~2297--2318.

\bibitem{BiBor09}
{\sc A. Biswas and V. S. Borkar}, {\em Small noise asymptotics for invariant densities for a class of diffusions: a control theoretic view},
J. Math. Anal. Appl. 360 (2009), pp.~476--484.

\bibitem{Day83}
{\sc M.~V. Day}, 
{\em On the exponential exit law in the small parameter exit problem},
Stochastics, 8 (1983), pp.~297--323.

\bibitem{Day87}
{\sc M.~V. Day}, 
{\em Recent progress on the small parameter exit problem},
Stochastics, 20 (1987), pp.~121--150.

\bibitem{FreWe84}
{\sc M.~I. Freidlin and A. D. Wentzell}, {\em Random perturbations of dynamical systems},
Springer, Berlin, 1984.

\bibitem{Fre06} 
{\sc A.~Friedman}, {\em Stochastic differential equations and applications},
Dover Publisher, Inc., Mineola, New York, 2006.

\bibitem{GilTr01} 
{\sc D.~Gilbarg and N. S. Trudinger}, {\em Elliptic partial differential equations of second order},
Springer Verlag, 2001.

\bibitem{Kif81} 
{\sc Y.~Kifer}, {\em The inverse problem for small random perturbations of dynamical systems},
Israel J. Math., 40 (1981), pp.~165--174.

\bibitem{QuaSi08}
{\sc A.~Quaas and B. Sirakov}, {\em Principal eigenvalue and the Dirichlet problem for fully nonlinear elliptic operators},
Advances in Math. 218 (2008), pp.~105--135.

\bibitem{VenFre70}
{\sc A.~D. Ventcel and M. I. Freidlin}, 
{\em On small random perturbations of dynamical systems},
Russian Math. Surv. 25 (1970), pp.~1--55.

\bibitem{Ven72}
{\sc A.~D. Ventcel}, {\em On the asymptotic behavior of the largest eigenvalue of a second-order elliptic differential operator with smaller parameter in the higher derivatives},
Soviet Math. Dokl., 13 (1972), pp.~13--17.

\end{thebibliography}
\end{document}